\documentclass[12pt,leqno]{article}

\usepackage{amsmath,amssymb,amsthm,amscd}
\usepackage[all]{xy}

\usepackage[top=25truemm,bottom=25truemm,left=27truemm,right=27truemm]{geometry}

\makeatletter
    
    \@addtoreset{equation}{section}
  \makeatother

\theoremstyle{plain}
\newtheorem{thm}{Theorem}[section]
\newtheorem{prop}[thm]{Proposition}
\newtheorem{lem}[thm]{Lemma}
\newtheorem{cor}[thm]{Corollary}

\theoremstyle{definition}
\newtheorem{defn}[thm]{Definition}
\newtheorem{exa}[thm]{Example}

\newcommand{\FRAC}[2]{\leavevmode\kern.1em\raise.5ex\hbox{\the\scriptfont0 #1}\kern-.1em/\kern-.15em\lower.25ex\hbox{\the\scriptfont0 #2}}

\newcommand{\mf}{\mathcal{F}}

\newcommand{\id}{\mathrm{id}}

\newcommand{\tr}{\mathrm{Tr}}
\newcommand{\btr}{\tr^{\mathrm{Br}}}
\newcommand{\mhom}{\mathcal{H}om}

\DeclareMathOperator{\Spec}{Spec}

\begin{document}
\title
{Having the same wild ramification is preserved\\ by the direct image}
\author{YURI YATAGAWA}
\date{}
\maketitle
\begin{abstract}
Let $S$ be the spectrum of an excellent henselian discrete valuation ring of residue characteristic $p$
and $X$ a separated scheme over $S$ of finite type.
Let $\Lambda$ and $\Lambda'$ be finite fields of characteristics $\ell\neq p$ and $\ell'\neq p$
respectively.
For elements $\mf\in K_{c}(X,\Lambda')$ and $\mf'\in K_{c}(X,\Lambda')$ of the Grothendieck groups
of constructible sheaves of $\Lambda$-modules and $\Lambda'$-modules on $X$
respectively, we introduce the notion that $\mf$ and $\mf'$ have the same wild ramification and prove that this condition is preserved by four of Grothendieck's six operations except the derived tensor product and $R\mhom$.
\end{abstract}
\section*{Introduction}

Let $X$ be a separated scheme over a field $k$ of characteristic $p$ of finite type and 
$\bar{X}$ a proper normal scheme over $k$ containing $X$ as a dense open subscheme.
Let $\mf$ and $\mf'$ be constructible complexes of $\Lambda$-modules on $X$,
where $\Lambda$ is a finite field of characteristic $\ell\neq p$.
Deligne-Illusie \cite{il} have given a sufficient condition for $\mf$ and $\mf'$ to have the same
Euler-Poincar\'{e} characteristic in terms of wild ramification of $\mf$ and $\mf'$.

Let $S$ be the spectrum of an excellent henselian discrete valuation ring of residue characteristic $p>0$ or the spectrum of a field of characteristic $p>0$.
Let $K_{c}(X,\Lambda)$ be the Grothendieck group of constructible sheaves of $\Lambda$-modules
on $X$.
Vidal \cite{vi1} has extended Deligne-Illusie's result to the case where $X$ is a separated scheme over $S$ of finite type
and where $\mf$ and $\mf'$ are elements of $K_{c}(X,\Lambda)$.
More precisely, Vidal has defined a subgroup $K_{c}(X,\Lambda)^{0}_{t}\subset K_{c}(X,\Lambda)$ 
called the Grothendieck group of constructible sheaves of virtual wild ramification 0 and
proved that $Rf_{!}\colon K_{c}(X,\Lambda)\rightarrow K_{c}(S,\Lambda)$ and $Rf_{*}\colon
K_{c}(X,\Lambda)\rightarrow K_{c}(S,\Lambda)$ induce
$Rf_{!}\colon K_{c}(X,\Lambda)^{0}_{t}\rightarrow K_{c}(S,\Lambda)^{0}_{t}$
and $Rf_{*}\colon K_{c}(X,\Lambda)^{0}_{t}\rightarrow K_{c}(S,\Lambda)^{0}_{t}$ respectively
for the structure morphism $f\colon X\rightarrow S$.
This result gives a sufficient condition for $\mf$ and $\mf'$ to have the same Swan conductor (\cite{vi1}).
Vidal \cite{vi2} has further extended this result to the case where 
$S$ is the spectrum of an excellent henselian discrete valuation ring of residue characteristic $p$
(possibly $p=0$)
and where $f$ is an $S$-morphism of separated schemes over $S$ of finite type.
 
In this paper, we give a definition of  
{\it the Grothendieck group $K_{c}(X,\Lambda)_{0}\subset K_{c}(X,\Lambda)$
of constructible sheaves of $\Lambda$-modules on $X$ of wild ramification $0$} 
in Definition \ref{defkczt}
along the notion that two constructible complexes have {\it the same wild ramification} introduced in \cite{sy}
and we prove an analogue of Vidal's result in \cite{vi2} for this group $K_{c}(X,\Lambda)_{0}$ 
and an $S$-morphism $f$ of separated schemes over $S$ of finite type.
More precisely, the group $K_{c}(X,\Lambda)_{0}$ is defined to be the subgroup
of $K_{c}(X,\Lambda)$ consisting of the elements which have the same wild ramification with $0$.
The same wild ramification condition is expressed in terms of the wild ramification of two complexes.
The difference from Deligne-Illusie's condition is that our condition is given in terms of the dimensions
of the fixed parts by elements of inertia groups of $p$-power orders instead of the Brauer traces
of these elements.
The subgroup $K_{c}(X,\Lambda)_{0}$ contains Vidal's subgroup $K_{c}(X,\Lambda)^{0}_{t}$
of $K_{c}(X,\Lambda)$.

Let $\Lambda'$ be another finite field of characteristic $\ell'\neq p$ and
let $\Delta_{c}(X,\Lambda,\Lambda')$ be the subgroup of $K_{c}(X,\Lambda)\times K_{c}(X,\Lambda')$, defined in Definition \ref{defswr},
consisting of the elements $(a,b)\in K_{c}(X,\Lambda)\times K_{c}(X,\Lambda')$ such that $a$ and $b$ have the same wild ramification.
The main theorem of this article is the following:

\begin{thm}
\label{thmmain}
Let $S$ be the spectrum of an excellent strict local henselian discrete valuation ring
of residue characteristic $p$ (possibly $p=0$). 
Let $f\colon X\rightarrow Y$ be an $S$-morphism 
of separated schemes over $S$ of finite type.
Let $f_{!}\times f_{!}\colon K_{c}(X,\Lambda)\times K_{c}(X,\Lambda')\rightarrow 
K_{c}(Y,\Lambda)\times K_{c}(Y,\Lambda')$ be the morphism induced by $Rf_{!}\times Rf_{!}$.
Then $f_{!}\times f_{!}$ induces $f_{!}\times f_{!}\colon \Delta_{c}(X,\Lambda, \Lambda')
\rightarrow \Delta_{c}(Y,\Lambda, \Lambda')$.
\end{thm}

The analogue of Vidal's result in \cite{vi2} for $K_{c}(X,\Lambda)_{0}$ 
is proved in Corollary \ref{corkzt} (iii) as a corollary of this theorem.
Theorem \ref{thmmain} also leads in Corollary \ref{corkzt} to
the compatibility of $K_{c}(X,\Lambda)_{0}$ with four of Grothendieck's six operations except 
the derived tensor product and $R\mhom$  
as well as Vidal's result in \cite{vi2}.
A partial result of Theorem \ref{thmmain}, which is under the assumption that $\dim Y\le 2$,
and a similar result for two complexes having the same Artin conductor when restricted to a curve 
have been obtained by Kato \cite{ka}.

We describe the construction of this paper.
In Section \ref{sbrtr}, we recall the Brauer trace.
We define the same wild ramification condition in Definition \ref{defswr} and 
give the definition of $K_{c}(X,\Lambda)_{0}$ in Definition \ref{defkczt}.
The proof of Theorem \ref{thmmain} is given in Section \ref{spfmain}.
In Section \ref{scor}, we give two corollaries of Theorem \ref{thmmain}
on the compatibility with Grothendieck's six operations.

The author would like to thank Professor Takeshi Saito for discussions.
This work is supported by JSPS KAKENHI Grant Number 15J03851.

\section{Brauer Trace}
\label{sbrtr}

We briefly recall the definition of the Brauer trace.
Let $G$ be a profinite group and $\Lambda$ a finite field of characteristic $\ell$.
Let $W(\Lambda)$ be the Witt ring of $\Lambda$.
The subgroup of $G$ consisting of $\ell$-regular elements is denoted by $G_{reg}$.
Let $K_{\cdot}(\Lambda[G])$ be the Grothendieck group of finite dimensional $\Lambda$-vector spaces 
with continuous $G$-actions.
Let $M$ be an element of $K_{\cdot}(\Lambda[G])$.
The Brauer trace $\btr_{M} \colon G_{reg}\rightarrow W(\Lambda)$ is a central function
of $G_{reg}$, and if $M$ is the class of a finite dimensional $\Lambda$-vector space with 
continuous $G$-action, then
it is given by $\btr_{M}(g)=\sum[\lambda]$ for $g\in G_{reg}$.
Here $\lambda$ runs through every eigenvalue of the action of $g$ on $M$
and $[\lambda]$ denotes the unique lift of $\lambda$ such that $[\lambda]$ is a root of unity
of order prime to $\ell$.
We note that the Brauer trace is an additive function, and is multiplicative with respect to the elements of $K_{\cdot}(\Lambda[G])$.

\begin{lem}[cf.\ {\cite[Lemma 4.1]{sy}}]
\label{lembtr}
Let $M$ be an element of $K_{\cdot}(\Lambda[G])$ and $g$ an element of $G$
of $p$-power order for a prime number $p$ different from $\ell$ as an automorphism of $M$.
For every subfield $E$ of the fractional field of $W(\Lambda)$ of finite degree over $\mathbf{Q}$
containing $\btr_{M}(g)$, we have
\begin{equation}
\label{eqbtre}
\frac{1}{[E:\mathbf{Q}]}\tr_{E/\mathbf{Q}}\btr_{M}(g)=\frac{1}{p-1}(p\cdot \dim M^{g}-\dim M^{g^{p}}). 
\end{equation}
\end{lem}

\begin{proof}
If $M$ is the class of a finite dimensional $\Lambda$-vector space
with continuous $G$-action,
the assertion follows from \cite[Lemma 4.1]{sy}.

Suppose that $M$ is the linear combination over $\mathbf{Z}$ of finitely many classes $\{M_{i}\}_{i}$ of finite dimensional 
$\Lambda$-vector spaces with continuous $G$-actions.
If $E$ contains $\btr_{M_{i}}(g)$ for every $i$, then the assertion follows
since $\tr_{E/\mathbf{Q}}$ is an additive function on $E$
and the Brauer trace is additive with respect to the elements of $K_{\cdot}(\Lambda[G])$.
If $E$ does not contain $\btr_{M_{i}}(g)$ for some $i$, take a finite extension $E'$ of $E$
containing $\btr_{M_{i}}(g)$ for every $i$.
Then we have
\begin{equation}
\label{eqbtrep}
\frac{1}{[E':\mathbf{Q}]}\tr_{E'/\mathbf{Q}}\btr_{M}(g)=\frac{1}{p-1}(p\cdot \dim M^{g}-\dim M^{g^{p}}) 
\end{equation}
by the case where $E$ contains $\btr_{M_{i}}(g)$ for every $i$.
Since $E$ contains $\btr_{M}(g)$, the left hand side of (\ref{eqbtrep}) is equal to that of (\ref{eqbtre}).
Hence the assertion follows.
\end{proof}

\begin{lem}
\label{lemdimcn}
Let $M$ and $N$ be elements of $K_{\cdot}(\Lambda[G])$ and $K_{\cdot}(\Lambda'[G])$ respectively.
Let $p$ be a prime
number different from $\ell$ and $\ell'$.
Let $g$ be an element of $G$ of $p$-power order.
Then the following are equivalent:
\begin{enumerate}
\item $p\cdot \dim M^{g^{p^{n}}} -\dim M^{g^{p^{n+1}}}=p\cdot \dim N^{g^{p^{n}}} -\dim N^{g^{p^{n+1}}}$ 
for every $n\in \mathbf{Z}_{\ge 0}$.
\item $\dim M^{g^{n}}=\dim N^{g^{n}}$ for every $n\in \mathbf{Z}_{\ge 0}$.
\end{enumerate}
\end{lem}

\begin{proof}
The condition (ii) obviously implies the condition (i).

Suppose that the condition (i) holds.
Then we have 
\begin{equation}
\label{eqmn}
p\cdot(\dim M^{g^{p^{n}}}-\dim N^{g^{p^{n}}})=\dim M^{g^{p^{n+1}}}-\dim N^{g^{p^{n+1}}}
\end{equation}
for every $n\in \mathbf{Z}_{\ge 0}$.
Let $s$ be an integer $\ge 0$ such that $g^{p^{s}}=\id_{M}$ and $g^{p^{s}}=\id_{N}$.
If $n\ge s$, we have $\dim M^{g^{p^{n}}}=\dim M^{g^{p^{n+1}}}=\dim M$ and 
$\dim N^{g^{p^{n}}}=\dim N^{g^{p^{n+1}}}=\dim N$.
Hence, by (i), we have
$(p-1)\cdot\dim M=(p-1)\cdot\dim N$.
Since $p-1\neq 0$ in $\mathbf{Q}$, we have $\dim M=\dim N$ and 
hence $\dim M^{g^{p^{n}}}=\dim N^{g^{p^{n}}}$ for every integer $n\ge s$.

If $0\le n <s$, then we have
$p^{s-n}\cdot(\dim M^{g^{p^{n}}}-\dim N^{g^{p^{n}}})=\dim M-\dim N$ by (\ref{eqmn}).
Since $\dim M=\dim N$ and $p^{s-n}\neq 0$ in $\mathbf{Q}$, the assertion follows.
\end{proof}

\section{Same wild ramification}
\label{sswr}
Let $S$ be an excellent trait of residue characteristic $p$, 
namely $S$ is the spectrum of an excellent henselian discrete valuation ring
of residue characteristic $p$.
Note that $p=0$ is admitted.
For example, a complete discrete valuation ring is an excellent henselian discrete valuation ring.
The generic point of $S$ is denoted by $\eta$ and the closed point of $S$ is denoted by $s$.

\begin{defn}[{\cite[Subsection 2.1]{vi1}, \cite[Section 1]{vi2}}]
\label{defezs}
Let $Z$ be a normal connected scheme over $S$ of finite type
and $\bar{v}_{0}$ a geometric generic point of $Z$.
\begin{enumerate}
\item Let $\bar{Z}$ be a normal compactification of $Z\rightarrow S$ containing $Z$ as a dense open subscheme,
namely $\bar{Z}$ is a proper normal scheme over $S$ containing $Z$ as a dense open subscheme.
Let $\bar{z}$ be a geometric point of $\bar{Z}$
and $\bar{v}$ a geometric point of $\bar{Z}_{(\bar{z})}\times_{\bar{Z}}Z$ lying above $\bar{v}_{0}$.
We define a subset $E_{\bar{z},\bar{v}}$ of $\pi_{1}(Z,\bar{v}_{0})$ to be the union of the image of 
$p$-Sylow subgroups of
$\pi_{1}(\bar{Z}_{(\bar{z})}\times_{\bar{Z}}Z,\bar{v})$,
where $\bar{Z}_{(\bar{z})}$ denotes the strict localization of $\bar{Z}$ at $\bar{z}$.
\item We define a subset $E_{Z/S,\bar{z}}$ of $\pi_{1}(Z,\bar{v}_{0})$ to be the union of conjugates of $E_{\bar{z},\bar{v}}$.
\item We define a subset $E_{Z/S,\bar{Z}}$ of $\pi_{1}(Z,\bar{v}_{0})$ to be the union of $E_{Z/S,
\bar{z}}$
for every geometric point $\bar{z}$ of $\bar{Z}$.
\item We define a subset $E_{Z/S}$ of $\pi_{1}(Z,\bar{v}_{0})$ 
to be the intersection of $E_{Z/S,\bar{Z}}$
for every normal compactification $\bar{Z}$ of $Z\rightarrow S$ containing $Z$ as a dense open subscheme.
\item Let $\tau$ be the generic point of $Z$.
We define a subset $E_{\tau/S}(Z)$ of $E_{Z/S}$ to be the intersection of the image of
$E_{U/S}\rightarrow E_{Z/S}$ for every dense open subscheme $U$ of $Z$. 
\end{enumerate}
\end{defn}

The definition of $E_{Z/S,\bar{v}}$ is independent of the choice of the geometric
point $\bar{v}$ of $\bar{Z}_{(\bar{z})}\times_{\bar{Z}}Z$ lying above $\bar{v}_{0}$.
In fact, for two geometric points $\bar{v}_{1}$ and 
$\bar{v}_{2}$ of $\bar{Z}_{(\bar{z})}\times_{\bar{Z}}Z$
lying above $\bar{v}_{0}$,
the subsets $E_{\bar{z},\bar{v}_{1}}$ and $E_{\bar{z},\bar{v}_{2}}$ of $\pi_{1}(Z,\bar{v}_{0})$
are conjugate.
We note that $E_{Z/S,\bar{z}}$, $E_{Z/S,\bar{Z}}$, and $E_{Z/S}$
are stable under conjugate.

Let $Q$ be a finite quotient of $\pi_{1}(Z,\bar{v}_{0})$.
Let $E_{Z/S,\bar{z}}(Q)$, $E_{Z/S,\bar{Z}}(Q)$, $E_{Z/S}(Q)$, and $E_{\tau/S}(Q)$ be the images of
$E_{Z/S,\bar{z}}$, $E_{Z/S,\bar{Z}}$, $E_{Z/S}$, and $E_{\tau/S}(Z)$ in $Q$ respectively.
Since the category of normal compactifications of $Z\rightarrow S$ containing $Z$ as a dense open
subscheme is cofiltered
and $Q$ is a finite group,
there exists a normal compactification $\bar{Z}$ of $Z\rightarrow S$ 
containing $Z$ as a dense open subscheme such that
$E_{Z/S}(Q)=E_{Z/S,\bar{Z}}(Q)$.
Since the category of dense open subschemes of $Z$ is cofiltered and $Q$ is a finite group,
there exists a dense open subscheme $U$ of $Z$ such that $E_{\tau/S}(Q)$ is the image of
$E_{U/S}$ in $Q$. 
\vspace{0.2cm}

Let $X$ be a separated scheme over $S$ of finite type.
Let $\Lambda$ and $\Lambda'$ be finite fields of characteristic $\ell\neq p$ and $\ell'\neq p$ respectively.
Let $K_{c}(X,\Lambda)$ be the Grothendieck group of constructible sheaves
of $\Lambda$-modules on $X$ and 
$K_{coh}(X,\Lambda)$ the subgroup of $K_{c}(X,\Lambda)$
generated by the classes of locally constant constructible sheaves of 
$\Lambda$-modules on $X$.
If $X$ is normal connected and if $\bar{t}_{0}$ is a geometric generic point of $X$, 
then the Grothendieck group $K_{coh}(X,\Lambda)$ is equal to 
the Grothendieck group $K_{\cdot}(\Lambda[\pi_{1}(X,\bar{t}_{0})])$
of finite dimensional $\Lambda$-vector spaces with continuous $\pi_{1}(X,\bar{t}_{0})$-actions.

\begin{lem}[{\cite[Lemme 2.1.1]{vi1}}] 
\label{lemvifet}
Let $f\colon X\rightarrow Y$ be an $S$-morphism of normal connected schemes over $S$
of finite type.
Let $\bar{t}_{0}$ be a geometric generic point of $X$ and
let $\bar{u}_{0}$ be a geometric generic point of $Y$. 
Let $\phi_{f}\colon \pi_{1}(X,\bar{t}_{0})\rightarrow \pi_{1}(Y,f(\bar{t}_{0})) \rightarrow \pi_{1}(Y,\bar{u}_{0})$
be the composition of the morphism $\pi_{1}(X,\bar{t}_{0})\rightarrow \pi_{1}(Y,f(\bar{t}_{0}))$
induced by $f$ and an isomorphism $\pi_{1}(Y,f(\bar{t}_{0})) \rightarrow \pi_{1}(Y,\bar{u}_{0})$
of fundamental groups.
Let $E_{X/S}$ and $E_{Y/S}$ be the subsets of $\pi_{1}(X,\bar{t}_{0})$ and $\pi_{1}(Y,\bar{u}_{0})$
respectively defined in Definition \ref{defezs} (iv).

\begin{enumerate}
\item $\phi_{f} (E_{X/S})\subset E_{Y/S}$.
\item If $f$ is finite \'{e}tale and if $\bar{u}_{0}=f(\bar{t}_{0})$, then we have $E_{X/S}=E_{Y/S}\cap \pi_{1}(X,\bar{t}_{0})$.
\end{enumerate}
\end{lem}

We introduce the notion of same wild ramification.

\begin{defn}
\label{defswrcoh}
Assume that $X$ is normal connected.
Let $\bar{t}_{0}$ be a geometric generic point of $X$ and 
let $E_{X/S}$ be the subset of $\pi_{1}(X,\bar{t}_{0})$
defined in Definition \ref{defezs} (iv). 
\begin{enumerate}
\item We put $A_{X/S}=\prod_{g\in E_{X/S}}\mathbf{Z}$.
We define the morphism $\varphi_{\Lambda}\colon K_{coh}(X,\Lambda)\rightarrow A_{X/S}$ of modules
by $\varphi_{\Lambda}(a)=(\dim a^{g})_{g}$.
\item We define a subgroup $\Delta_{coh}(X,\Lambda,\Lambda')$ of $K_{coh}(X,\Lambda)\times
K_{coh}(X,\Lambda')$ to be the 
kernel of the morphism 
\begin{equation}
\varphi_{\Lambda}-\varphi_{\Lambda'}
\colon K_{coh}(X,\Lambda)\times K_{coh}(X,\Lambda')\rightarrow A_{X/S}
\; ;\; (a,b)\mapsto \varphi_{\Lambda}(a)-\varphi_{\Lambda'}(b). \notag
\end{equation}
\end{enumerate}
\end{defn}

Let $\bar{t}_{0}$ and $\bar{t}_{0}'$ be geometric generic points of $X$.
Since the fundamental groups $\pi_{1}(X,\bar{t}_{0})$ and $\pi_{1}(X,\bar{t}_{0}')$ are isomorphic
and since $E_{X/S}\subset \pi_{1}(X,\bar{t}_{0})$ and $E_{X/S}\subset \pi_{1}(X,\bar{t}_{0}')$
are isomorphic by the isomorphism of $\pi_{1}(X,\bar{t}_{0})$ and $\pi_{1}(X,\bar{t}_{0}')$,
the definition of $\Delta_{coh}(X,\Lambda,\Lambda')$ 
is independent of the choice of the geometric generic point $\bar{t}_{0}$ of $X$.
By this reason, we sometimes omit the base point $\bar{t}_{0}$ of $\pi_{1}(X,\bar{t}_{0})$ 
when we consider $\Delta_{coh}(X,\Lambda,\Lambda')$.

\begin{defn}
\label{defswr}
Let the notation be as before Lemma \ref{lemvifet}.
\begin{enumerate}
\item We define a subgroup $\Delta_{c}(X,\Lambda,\Lambda')$ of $K_{c}(X,\Lambda)\times K_{c}(X,
\Lambda')$ to be the subgroup of $K_{c}(X,\Lambda)\times K_{c}(X,\Lambda')$
consisting of the pairs $(a,b)$ such that 
there exists a finite decomposition $X=\coprod_{i}X_{i}$ of $X$
into normal connected locally closed subschemes $\{X_{i}\}_{i}$ of $X$ such that
$(a|_{X_{i}},b|_{X_{i}})\in \Delta_{coh}(X_{i},\Lambda,\Lambda')$ for every $i$.
\item Let $a$ and $b$ be elements of $K_{c}(X,\Lambda)$ and $K_{c}(X,\Lambda')$ respectively.
We say that $a$ and $b$ have {\it the same wild ramification} 
if $(a,b)\in \Delta_{c}(X,\Lambda,\Lambda')$. 
\end{enumerate}
\end{defn}

Let $(a,b)$ be an element of $K_{c}(X,\Lambda)\times K_{c}(X,\Lambda')$.
By the induction on the dimension of $X$, there always exists 
a finite decomposition $X=\coprod_{i}X_{i}$ of $X$
into normal connected locally closed 
subschemes $\{X_{i}\}_{i}$ of $X$ such that $(a|_{X_{i}},b|_{X_{i}})\in K_{coh}(X_{i},\Lambda)
\times K_{coh}(X_{i},\Lambda')$ for every $i$.
If $\Lambda=\Lambda'$, then we have $(a,b)\in \Delta_{coh}(X,\Lambda,\Lambda)$
if and only if $(b,a)\in \Delta_{coh}(X,\Lambda,\Lambda)$,
and we have $(a,b)\in \Delta_{c}(X,\Lambda,\Lambda)$
if and only if $(b,a)\in \Delta_{c}(X,\Lambda,\Lambda)$.

\begin{lem}
\label{lemeqcn}
Let $a$ and $b$ be elements of $K_{c}(X,\Lambda)$ and $K_{c}(X,\Lambda')$ respectively.
\begin{enumerate}
\item Assume that $a\in K_{coh}(X,\Lambda)$ and $b\in K_{coh}(X,\Lambda')$.
Then the following are equaivalent:
\begin{enumerate}
\item $(a,b)\in \Delta_{coh}(X,\Lambda, \Lambda')$.
\item $a$ and $b$ satisfy the following condition for a normal compactification $\bar{X}$ of
$X\rightarrow S$ containing $X$ as a dense open subscheme and for every geometric point
$\bar{x}$ of $\bar{X}$:
\begin{itemize}
\item[(W)] Let $M$ and $N$ be the elements of 
$K_{\cdot}(\Lambda[G])$ and $K_{\cdot}(\Lambda'[G])$
corresponding to the pull-backs of $a$ and $b$ to
$\bar{X}_{(\bar{x})}\times_{\bar{X}}X$ respectively, where $G$ is a finite quotient of the inertia group
$I_{\bar{x}}=\pi_{1}(\bar{X}_{(\bar{x})}\times_{\bar{X}}X,\bar{t})$ 
for a geometric point $\bar{t}$ of $\bar{X}_{(\bar{x})}\times_{\bar{X}}X$.
For every element $g\in G$ of $p$-power order, we have $\dim M^{g}=\dim N^{g}$.
\end{itemize}
\end{enumerate}
\item The following are equivalent:
\begin{enumerate}
\item $(a,b)\in \Delta_{c}(X,\Lambda,\Lambda')$, namely $a$ and $b$ have the same wild ramification.
\item There exists a finite decomposition $X=\coprod_{i}X_{i}$ of $X$ into normal connected locally closed 
subschemes $\{X_{i}\}_{i}$ of $X$ and normal compactifications $\{\bar{X}_{i}\}_{i}$ of $\{X_{i}\rightarrow
S\}_{i}$ containing $\{X_{i}\}_{i}$ as dense open subschemes respectively
such that 
$(a|_{X_{i}},b|_{X_{i}})\in K_{coh}(X_{i},\Lambda)\times K_{coh}(X_{i},\Lambda')$ for every $i$
and that $a|_{X_{i}}$ and $b|_{X_{i}}$ satisfy the condition (W) in (i) (b) for every $i$ and 
every geometric point $\bar{x}$ of $\bar{X}_{i}$.
\end{enumerate}
\end{enumerate}
\end{lem}

\begin{proof}
The assertion (ii) follows from (i).
We prove (i).
Let $Q$ be a finite quotient of $\pi_{1}(X)$ through which $\pi_{1}(X)$ acts $a$ and $b$.
Take a normal compactification $\bar{X}$ of $X\rightarrow S$ containing $X$ as a dense open subscheme such that $E_{X/S}(Q)=E_{X/S,\bar{X}}(Q)$.
Then (a) implies (b) for the compactification $\bar{X}$.
Suppose that the condition (b) holds.
Let $\bar{x}$ be a geometric point of $\bar{X}$ and $\bar{t}_{0}$ a geometric generic point of $X$.
Suppose that $\bar{t}$ is lying above a geometric point $\bar{t}_{0}'$ of $X$.
Then $E_{X/S,\bar{x}}$ is equal to the union of conjugates of the union of images of
$p$-Sylow subgroups of $I_{\bar{x}}$
by an isomorphism $\pi_{1}(X,\bar{t}_{0}')\rightarrow \pi_{1}(X,\bar{t}_{0})$.
Hence (b) implies (a).
\end{proof}

For an $S$-morphism $f\colon X\rightarrow Y$ of separated schemes over $S$ of finite type, 
let $f_{!}\colon K_{c}(X,\Lambda)\rightarrow K_{c}(Y,\Lambda)$
and $f^{*}\colon K_{c}(Y,\Lambda)\rightarrow K_{c}(X,\Lambda)$
denote the morphisms induced by
the functors $Rf_{!}$ and $f^{*}$ respectively.

\begin{prop}
\label{propfct}
Let $f\colon X\rightarrow Y$ be an $S$-morphism of 
separated schemes over $S$ of finite type.
\begin{enumerate}
\item $f^{*}\times f^{*}\colon K_{c}(Y,\Lambda)\times K_{c}(Y,\Lambda')\rightarrow
K_{c}(X,\Lambda)\times K_{c}(X,\Lambda')$ induces $f^{*}\times f^{*}\colon
\Delta_{c}(Y,\Lambda,\Lambda')\rightarrow \Delta_{c}(X,\Lambda,\Lambda')$.
\item If $f$ is quasi-finite, then 
$f_{!}\times f_{!}\colon K_{c}(X,\Lambda)\times K_{c}(X,\Lambda')\rightarrow
K_{c}(Y,\Lambda)\times K_{c}(Y,\Lambda')$ induces $f_{!}\times f_{!}\colon
\Delta_{c}(X,\Lambda,\Lambda')\rightarrow \Delta_{c}(Y,\Lambda,\Lambda')$.
\end{enumerate}
\end{prop}

\begin{proof}
The proof goes similarly as the proof of \cite[Proposition 2.3.3 (i), (ii)]{vi1}.

(i) Let $(a,b)$ be an element of $\Delta_{c}(Y,\Lambda,\Lambda')$. 
Take a finite decomposition $Y=\coprod_{i}Y_{i}$ of $Y$
into normal connected locally closed subschemes $\{Y_{i}\}_{i}$ of $Y$
such that $(a|_{Y_{i}},b|_{Y_{i}})\in \Delta_{coh}(Y_{i},\Lambda,\Lambda')$ for every $i$.
We put $X_{i}=f^{-1}(Y_{i})$ and let $X_{i}=\coprod_{j}X_{ij}$ be a finite decomposition of $X_{i}$ into
normal connected locally closed subschemes $\{X_{ij}\}_{j}$ of $X_{i}$ for $i$.
Let $f_{ij}\colon X_{ij}\rightarrow Y_{i}$ be the morphism induced by $f$ for $i$ and $j$. 
Then, by Lemma \ref{lemvifet} (i), we have $(f_{ij}^{*}(a|_{Y_{i}}),f^{*}_{ij}(b|_{Y_{i}}))
\in \Delta_{coh}(X_{ij},\Lambda,\Lambda')$ for every $i$ and $j$.
Hence the assertion follows.

(ii) Let $(a,b)$ be an element of $\Delta_{c}(X,\Lambda,\Lambda')$.  
By decomposing $Y$ into the disjoint union of connected components of $Y$,
we may assume that $Y$ is connected.
By Zariski's main theorem, it is sufficient to prove the assertion in the case where $f$ is an open 
immersion or a finite morphism.
Suppose that $f$ is an open immersion.
Take the decomposition $Y=X\amalg (Y\setminus X)$.
Since $(f_{!}a)|_{X}=a$ and $(f_{!}b)|_{X}=b$, it follows that $(f_{!}a)|_{X}$ and $(f_{!}b)|_{X}$
have the same wild ramification.
Since $(f_{!}a)|_{Y\setminus X}=0$ and $(f_{!}b)|_{Y\setminus X}=0$, 
by taking
a finite decomposition of $Y\setminus X$ into normal connected locally closed subschemes of $Y\setminus X$,
we see that $(f_{*}a)|_{Y\setminus X}$ and $(f_{*}b)|_{Y\setminus X}$ have the same wild ramification.
Hence the assertion follows if $f$ is an open immersion.

Assume that $f$ is a finite morphism.
Then we have $f_{!}a=f_{*}a$ and $f_{!}b=f_{*}b$.
By Lemma \ref{lemeqcn} (ii) and the induction on the dimension of $Y$, it is sufficient to prove that
there exists a normal dense open subscheme $U$ of $Y$
such that $((f_{*}a)|_{U},(f_{*}b)|_{U})\in \Delta_{coh}(U,\Lambda,\Lambda')$.
By shrinking $Y$ to a normal affine dense open subscheme if necessary,
we may assume that $Y$ is normal affine and thus that $X$ is affine.
Since $f_{*}a$ and $f_{*}b$ are constructible, we may assume that 
$(f_{*}a,f_{*}b)\in K_{coh}(Y,\Lambda)\times K_{coh}(Y,\Lambda')$
by shrinking $Y$ to a dense open subscheme.
Let $X=\coprod_{i}X_{i}$ be a finite decomposition of $X$ into normal connected locally 
closed subschemes $\{X_{i}\}_{i}$ of $X$ such that $(a|_{X_{i}},b|_{X_{i}}) \in \Delta_{coh}(X_{i},
\Lambda,\Lambda')$ for every $i$.
Let $i_{X_{i}}\colon X_{i}\rightarrow X$ be the immersion for $i$.
Since $a=\sum_{i}i_{X_{i}!}(a|_{X_{i}})$ and $b=\sum_{i}i_{X_{i}!}(b|_{X_{i}})$, 
we may replace $f$, $a$, and $b$ by $f\circ i_{X_{i}}$, $a|_{X_{i}}$, and $b|_{X_{i}}$ respectively.
Hence we may assume that $X$ is normal connected and that
$(a,b)\in \Delta_{coh}(X,\Lambda,\Lambda')$.

Since $f$ is factorized to the composition of a finite surjective morphism and a closed immersion,
we may assume that $f$ is a finite surjective morphism or a closed immersion.
Suppose that $f$ is a closed immersion.
Then we have $(f_{*}a)|_{X}=a$ and $(f_{*}b)|_{X}=b$.
Further we have $(f^{*}a)|_{Y\setminus X}=0$ and $(f_{*}b)|_{Y\setminus X}=0$.
Hence the assertion follows similarly as the case where $f$ is an open immersion.
Therefore we may assume that $f$ is a finite surjective morphism.
Let $F_{X}$ and $F_{Y}$ be the function fields of $X$ and $Y$ respectively.
Let $F^{s}$ be the separable closure of $F_{Y}$ in $F_{X}$.
By replacing $X$ by the normalization of $Y$ in $F^{s}$ if necessary, 
we may assume that $f$ is generically \'{e}tale.
By shrinking $Y$ to a dense open subscheme if nesessary, we may assume that $f$ is finite \'{e}tale. 
Then $f_{*}a$ and $f_{*}b$ are $a\otimes_{\Lambda[\pi_{1}(X)]}\Lambda[\pi_{1}(Y)]$ and
$b\otimes_{\Lambda'[\pi_{1}(X)]}\Lambda'[\pi_{1}(Y)]$ respectively.
Further, by Lemma \ref{lemvifet} (ii), we have $E_{X/S}=E_{Y/S}\cap \pi_{1}(X)$.

Let $g$ be an element of $E_{Y/S}$ and let $R$ be a representative system of $\pi_{1}(Y)/\pi_{1}(X)$.
By \cite[Exercise 18.2]{se}, we have
$\btr_{f_{*}a}(g)=\sum_{r\in R, rgr^{-1}\in \pi_{1}(X)}\btr_{a}(rgr^{-1})$ and  
$\btr_{f_{*}b}(g)=\sum_{r\in R, rgr^{-1}\in \pi_{1}(X)}\btr_{b}(rgr^{-1})$.
Let $r$ be an element of $R$ such that $rgr^{-1}\in \pi_{1}(X)$.
Since $E_{X/S}=E_{Y/S}\cap \pi_{1}(X)$, we have $rgr^{-1}\in E_{X/S}$.
Let $E$ be a subfield of the fractional field of $W(\Lambda)$ of finite degree over $\mathbf{Q}$
containing $\btr_{a}(rgr^{-1})$ 
for every $r\in R$ such that $rgr^{-1}\in \pi_{1}(X)$
and let $E'$ be a subfield of the fractional field of $W(\Lambda')$ 
of finite degree over $\mathbf{Q}$ containing $\btr_{b}(rgr^{-1})$ for every $r\in R$ such that $rgr^{-1}\in \pi_{1}(X)$.
Since $\dim a^{g'}=\dim b^{g'}$ for every $g'\in E_{X/S}$, 
we have 
\begin{equation}
\frac{1}{[E:\mathbf{Q}]}\tr_{E/\mathbf{Q}}\btr_{a}(rgr^{-1})=
\frac{1}{[E':\mathbf{Q}]}\tr_{E'/\mathbf{Q}}\btr_{b}(rgr^{-1}) \notag
\end{equation}
for every $r\in R$ such that $rgr^{-1}\in \pi_{1}(X)$ by Lemma \ref{lembtr}.
Since $[E:\mathbf{Q}]$ and $[E':\mathbf{Q}]$ are positive integers
and since $\tr_{E/\mathbf{Q}}$ and $\tr_{E'/\mathbf{Q}}$ are additive functions, we have 
\begin{equation}
\frac{1}{[E:\mathbf{Q}]}\tr_{E/\mathbf{Q}}\btr_{f_{*}a}(g)=
\frac{1}{[E':\mathbf{Q}]}\tr_{E'/\mathbf{Q}}\btr_{f_{*}b}(g). \notag
\end{equation}
Hence the assertion follows by Lemma \ref{lembtr} and Lemma \ref{lemdimcn}.
\end{proof}

\begin{defn}[cf.\ {\cite[D\'{e}finition 2.3.1]{vi1}}]
\label{defkczt}
We define a subgroup $K_{c}(X,\Lambda)_{0}$ of $K_{c}(X,\Lambda)$ 
to be the subgroup of $K_{c}(X,\Lambda)$ consisting the elements $a\in K_{c}(X,\Lambda)$
such that $a$ and $0$ have the same wild ramification with $\Lambda'=\Lambda$.
We call the subgroup $K_{c}(X,\Lambda)_{0}$ {\it the Grothendieck group of
constructible sheaves of $\Lambda$-modules on $X$ of wild ramification $0$}.
\end{defn}

\begin{cor}[cf.\ {\cite[Proposition 2.3.3]{vi1}}]
Let $f\colon X\rightarrow Y$ be an $S$-morphism of 
separated schemes over $S$ of finite type.
\begin{enumerate}
\item The functor $f^{*}\colon K_{c}(Y,\Lambda)\rightarrow K_{c}(X,\Lambda)$ induces
$f^{*}\colon K_{c}(Y,\Lambda)_{0}\rightarrow K_{c}(X,\Lambda)_{0}$.
\item If $f$ is quasi-finite, then the functor $f_{!}\colon K_{c}(X,\Lambda)\rightarrow K_{c}(Y,\Lambda)$
induces $f_{!}\colon K_{c}(X,\Lambda)_{0}\rightarrow K_{c}(Y,\Lambda)_{0}$.
\item Let $X=\coprod_{i}X_{i}$ be a finite decomposition into locally closed subschemes $\{X_{i}\}_{i}$
of $X$. Then we have $K_{c}(X,\Lambda)_{0}=\bigoplus_{i}K_{c}(X_{i},\Lambda)_{0}$.
\end{enumerate}
\end{cor}

\begin{proof}
Since $\Delta_{c}(X,\Lambda,\Lambda)\cap (K_{c}(X,\Lambda)\times \{0\})$ and
$\Delta_{c}(Y,\Lambda,\Lambda)\cap (K_{c}(Y,\Lambda)\times \{0\})$ are isomorphic to
$K_{c}(X,\Lambda)_{0}$ and $K_{c}(Y,\Lambda)_{0}$ respectively by the first projections,
the assertions (i) and (ii) follow from Proposition \ref{propfct} (i) and (ii) respectively.
We prove (iii).
Let $a$ be an element of $K_{c}(X,\Lambda)$ and let
$i_{X_{i}}\colon X_{i}\rightarrow X$ be the immersion for $i$.
Since $a=\sum_{i}i_{X_{i}!}(a|_{X_{i}})$ and $i_{X_{i}}$ is quasi-finite for every $i$,
the assertion follows by (i) and (ii).
\end{proof}

\section{Proof of Theorem \ref{thmmain}}
\label{spfmain}
Let the notation be as in Section \ref{sswr}.
We devote this section to the proof of Theorem \ref{thmmain}. 
The proof goes similarly as Vidal's proof of \cite[Th\'{e}or\`{e}me 0.1]{vi2}.

Let $(a,b)$ be an element of $\Delta_{c}(X,\Lambda,\Lambda')$
and assume that $S$ is strict local.
Let $f\colon X\rightarrow Y$ be an $S$-morphism of separated schemes over $S$ of finite type.
Take the decompositions $X=X_{\eta}\amalg X_{s}$ and $Y=Y_{\eta}\amalg Y_{s}$, and
decompose $f$ into $f_{\eta}\oplus f_{s}$, where $f_{\eta}\colon X_{\eta}\rightarrow Y_{\eta}$
and $f_{s}\colon X_{s}\rightarrow Y_{s}$ are induced by $f$.
Let $i_{Y_{\eta}}\colon Y_{\eta}\rightarrow Y$ and $i_{Y_{s}}\colon Y_{s}\rightarrow Y$
be immersions.
Since $f_{!}a=i_{Y_{\eta}!}f_{\eta!}(a|_{X_{\eta}})+i_{Y_{s}!}f_{s!}(a|_{X_{s}})$
and $f_{!}b=i_{Y_{\eta}!}f_{\eta!}(b|_{X_{\eta}})+i_{Y_{s}!}f_{s!}(b|_{X_{s}})$,
we may replace $f$ by $f_{\eta}$ or $f_{s}$ by Proposition \ref{propfct}.

Let $z$ be the generic point $\eta$ of $S$ or the closed point $s$ of $S$. 
Since the assertion is local, we may assume that $Y=\Spec B$ is affine.
Further, by taking a finite decomposition of $X$ into locally closed affine
subschemes of $X$ if necessary and applying Proposition \ref{propfct} (i), 
we may assume that $X=\Spec A$ is affine.
We identify $A$ with $B[t_{1},\ldots,t_{d}]/I$ for some $d\in \mathbf{Z}_{\ge 0}$ and an ideal $I$
of $B[t_{1},\ldots, t_{d}]$.
Since the assertion follows if $f$ is a closed immersion by Proposition \ref{propfct} (ii),
by taking the factorization $B\rightarrow B[t_{1}]\rightarrow \cdots \rightarrow B[t_{1},\ldots,t_{d}]
\rightarrow B[t_{1},\ldots,t_{d}]/I=A$ of $f^{\sharp}\colon B\rightarrow A$, 
we may assume that $f$ is (smooth) of relative dimension $1$. 
Since $f_{!}a$ and $f_{!}b$ are constructible, as in the proof of Proposition \ref{propfct} (ii),
we may assume that $Y$ is normal connected and 
that $(f_{!}a,f_{!}b)\in K_{coh}(Y,\Lambda)\times K_{coh}(Y,\Lambda')$.
Further, as in the proof of Proposition \ref{propfct} (ii), 
we may assume that $X$ is normal connected and 
that $(a,b)\in \Delta_{coh}(X,\Lambda,\Lambda')$.

Let $p\colon V\rightarrow X$ be a galois \'{e}tale covering 
trivializing $a$ and $b$ with galois group $H$.
By shrinking $Y$ to a dense open subscheme if necessary, 
we may assume that $(f\circ p)_{!}\Lambda\in K_{coh}(Y,\Lambda)$
and $(f\circ p)_{!}\Lambda'\in K_{coh}(Y,\Lambda')$.
Let $\tau$ be the generic point of $Y$.
We apply the following proposition proved by Vidal in \cite[Section 3]{vi2}:

\begin{prop}[{\cite[Corollary 3.0.5 and the proof of Th\'{e}or\`{e}me 0.1]{vi2}}]
\label{propvitr}
Let $f\colon X\rightarrow Y$ be a $z$-morphism of relative dimension $\le 1$
of normal affine connected schemes over $z$ of finite type.
Let $a$ be an element of $K_{coh}(X,\Lambda)$
and let $p\colon V\rightarrow X$ be a galois \'{e}tale covering trivializing $a$
with galois group $H$.
Let $E_{\lambda}$ be a finite extension of $\mathbf{Q}_{\ell}$ whose integer ring has
the residue field $\Lambda$.
Assume that $f_{!}a\in K_{coh}(Y,\Lambda)$ and $(f\circ p)_{!} \Lambda \in K_{coh}(Y,\Lambda)$.
Let $g$ be an element of $E_{\tau/S}(Y)$
and let $P$ be the subgroup of $\pi_{1}(Y)$ generated by $g$.
Let $P'$ be a pro-$p$-subgroup of $Gal(\bar{\tau}/\tau)$ whose image by the canonical
morphism $Gal(\bar{\tau}/\tau)\rightarrow \pi_{1}(Y)$ is $P$.
We put $H'=P'\times H$.
Then we have

\begin{equation}
\btr_{f_{!}a}(g)=\frac{1}{|H|}\sum_{\substack{h'\in H' \\ p_{P'}(h')=g'}}
\tr_{R\Gamma_{c}(V_{\bar{\tau}},E_{\lambda})}(h')\times
\btr_{a}(p_{H}(h')), \notag
\end{equation}
where $p_{P'}\colon H'\rightarrow P'$ and $p_{H}\colon H'\rightarrow H$
are projections and $g'\in P'$ is a lift of $g$.  
\end{prop}

Let $Q$ be a finite quotient of $\pi_{1}(Y)$ through which $\pi_{1}(Y)$ acts on $f_{!}a$, $f_{!}b$,
$(f\circ p)_{!}\Lambda$, and $(f\circ p)_{!}\Lambda'$.
By shrinking $Y$ to a dense open subscheme if necessary,
we may assume that $E_{Y/S}(Q)=E_{\tau/S}(Q)$.
Hence we may replace $E_{Y/S}$ by $E_{\tau/S}$.
Let $g$ be an element of $E_{Y/S}$.
With the notation in Proposition \ref{propvitr},
we have
\begin{equation}
\btr_{f_{!}a}(g)=\frac{1}{|H|}\sum_{\substack{h'\in H' \\ p_{P'}(h')=g'}}
\tr_{R\Gamma_{c}(V_{\bar{\tau}},E_{\lambda})}(h')\times
\btr_{a}(p_{H}(h')).\notag
\end{equation}
Let $E_{\lambda}'$ be a finite extension of $\mathbf{Q}_{\ell'}$ whose integer ring has
the residue field $\Lambda'$.
With the notation in Proposition \ref{propvitr}, we have
\begin{equation}
\btr_{f_{!}b}(g)=\frac{1}{|H|}\sum_{\substack{h'\in H' \\ p_{P'}(h')=g'}}
\tr_{R\Gamma_{c}(V_{\bar{\tau}},E_{\lambda}')}(h')\times
\btr_{b}(p_{H}(h')). \notag
\end{equation}
Further we apply the following proposition proved by Vidal in \cite[Section 3]{vi2}:

\begin{prop}[{\cite[Corollaire 3.0.7 and the proof of Th\'{e}or\`{e}me 0.1]{vi2}}]
\label{propvip}
Let the notation and the assumption be as in Proposition \ref{propvitr}.
Let $h'$ be an element of $H'$ such that 
$\tr_{R\Gamma_{c}(V_{\tau},E_{\lambda})}(h')\neq 0$. 
Then $p_{H}(h')$ is the image of an element of $E_{X/S}$ by the surjection
$\pi_{1}(X)\rightarrow H$.
\end{prop}

Let $E$ be a subfield of the fractional field of $W(\Lambda)$ of finite degree over $\mathbf{Q}$
containing $\btr_{a}(p_{H}(h'))$ for every $h'\in H'$ such that $p_{P'}(h')=g'$ and 
$\tr_{R\Gamma_{c}(V_{\bar{\tau}},E_{\lambda})}(h')\neq 0$.
Further, 
let $E'$ be a subfield of the fractional field of $W(\Lambda')$ of finite degree over $\mathbf{Q}$
containing $\btr_{b}(p_{H}(h'))$ for every $h'\in H'$ such that $p_{P'}(h')=g'$
and $\tr_{R\Gamma_{c}(V_{\bar{\tau}},E_{\lambda}')}(h')\neq 0$.
By \cite[Proposition 2.2.1]{vi2} and \cite[Proposition 2.3.1]{vi2},
the trace $\tr_{R\Gamma_{c}(V_{\tau},E_{\lambda})}(h')$ is an integer independent of $\ell\neq p$
and the extension $E_{\lambda}$ of $\mathbf{Q}_{\ell}$
for every $h'\in H'$.
By Lemma \ref{lembtr} and Proposition \ref{propvip}, we have
\begin{equation}
\frac{1}{[E:\mathbf{Q}]}\tr_{E/\mathbf{Q}}\btr_{a}(p_{K}(h'))=
\frac{1}{[E':\mathbf{Q}]}\tr_{E'/\mathbf{Q}}\btr_{b}(p_{K}(h')) \notag
\end{equation}
for every $h'\in K'$ such that $p_{P'}(h')=g'$ and that $\tr_{R\Gamma_{c}(V_{\bar{\tau}},E_{\lambda})}(h')\neq 0$.
Since $[E:\mathbf{Q}]$, $[E':\mathbf{Q}]$, and $|H|$ are positive integers,
we have 
\begin{equation}
\frac{1}{[E:\mathbf{Q}]}\tr_{E/\mathbf{Q}}\btr_{f_{!}a}(g)=
\frac{1}{[E':\mathbf{Q}]}\tr_{E'/\mathbf{Q}}\btr_{f_{!}b}(g). \notag
\end{equation}
Hence the assertion follows by Lemma \ref{lembtr} 
and Lemma \ref{lemdimcn}.

\section{Compatibility with functors}
\label{scor}

Let $S$ be an excellent trait of residue characteristic $p$.
Let $\Lambda$ and $\Lambda'$ be finite fields of characteristic $\ell\neq p$ and $\ell'\neq p$ respectively.
We prove two corollaries of Theorem \ref{thmmain}.

For an $S$-morphism $f\colon X\rightarrow Y$ of separated schemes over $S$ of finite type, 
let $f_{*}\colon K_{c}(X,\Lambda)\rightarrow K_{c}(Y,\Lambda)$ and 
$f^{!}\colon K_{c}(Y,\Lambda)\rightarrow K_{c}(X,\Lambda)$ 
denote the morphisms induced by
the functors $Rf_{*}$ and $Rf^{!}$ respectively.
Let $D_{X}\colon K_{c}(X,\Lambda)\rightarrow K_{c}(X,\Lambda)$ be the morphism
induced by the dualizing functor (\cite[4.2]{sga4h}). 

\begin{cor}
\label{corswr}
Let $f\colon X\rightarrow Y$ be an $S$-morphism 
of separated schemes over $S$ of finite type.
\begin{enumerate}
\item $f^{*}\times f^{*}\colon K_{c}(Y,\Lambda)\times K_{c}(Y,\Lambda')\rightarrow
K_{c}(X,\Lambda)\times K_{c}(X,\Lambda')$ induces $f^{*}\times f^{*}\colon
\Delta_{c}(Y,\Lambda,\Lambda')\rightarrow \Delta_{c}(X,\Lambda,\Lambda')$.
\item $f_{*}\times f_{*}\colon K_{c}(X,\Lambda)\times K_{c}(X,\Lambda')\rightarrow
K_{c}(Y,\Lambda)\times K_{c}(Y,\Lambda')$ induces $f_{*}\times f_{*}\colon
\Delta_{c}(X,\Lambda,\Lambda')\rightarrow \Delta_{c}(Y,\Lambda,\Lambda')$.
\item $f_{!}\times f_{!}\colon K_{c}(X,\Lambda)\times K_{c}(X,\Lambda')\rightarrow
K_{c}(Y,\Lambda)\times K_{c}(Y,\Lambda')$ induces $f_{!}\times f_{!}\colon
\Delta_{c}(X,\Lambda,\Lambda')\rightarrow \Delta_{c}(Y,\Lambda,\Lambda')$.
\item $f^{!}\times f^{!}\colon K_{c}(Y,\Lambda)\times K_{c}(Y,\Lambda')\rightarrow
K_{c}(X,\Lambda)\times K_{c}(X,\Lambda')$ induces $f^{!}\times f^{!}\colon
\Delta_{c}(Y,\Lambda,\Lambda')\rightarrow \Delta_{c}(X,\Lambda,\Lambda')$.
\item $D_{X}\times D_{X}\colon K_{c}(X,\Lambda)\times K_{c}(X,\Lambda')\rightarrow
K_{c}(X,\Lambda)\times K_{c}(X,\Lambda')$ induces $D_{X}\times D_{X}\colon
\Delta_{c}(X,\Lambda,\Lambda')\rightarrow \Delta_{c}(X,\Lambda,\Lambda')$.
\end{enumerate}
\end{cor}

\begin{proof}
We have already proved (i) in Proposition \ref{propfct} (i).

We prove (ii) and (iii).
By \cite[Th\'{e}or\`{e}me 1.1]{la}, if $S$ is strict local then we have $f_{!}=f_{*}$
as a morphism of Grothendieck groups.
Let $(a,b)$ be an element of $\Delta_{c}(X,\Lambda,\Lambda')$ and 
let $\pi\colon S^{sh}\rightarrow S$ be the strict localization.
We use the same notation $\pi$ for the base change $X_{S^{sh}}\rightarrow X$ 
of $X$ by $\pi\colon S^{sh}\rightarrow S$.
Since $(\pi^{*}a,\pi^{*}b)\in \Delta_{c}(X_{S^{sh}},\Lambda,\Lambda')$
and $\pi^{*}\circ f_{!}=f_{!}\circ \pi^{*}$,
we have $(\pi^{*}f_{!}a,\pi^{*}f_{!}b)\in \Delta_{c}(Y_{S^{sh}},\Lambda,\Lambda')$
by Theorem \ref{thmmain}.
Since $\pi^{*}\circ f_{*}=f_{*}\circ\pi^{*}=f_{!}\circ\pi^{*}$,
we have $(\pi^{*}f_{*}a,\pi^{*}f_{*}b)\in \Delta_{c}(Y_{S^{sh}},\Lambda,\Lambda')$.
Since the assertion is \'{e}tale local, the assertions (ii) and (iii) follow.

We prove (iv).
Since the assertion is local, we may assume that $X=\Spec A$ and $Y=\Spec B$ are affine.
We identify $A$ with $B[t_{1},\ldots,t_{d}]/I$ for some $d\in \mathbf{Z}_{\ge 0}$ and an ideal $I$ of
$B[t_{1},\ldots,t_{n}]$. 
Then the morphism $h$ is factorized to the composition of a closed immersion
$X\rightarrow \mathbf{A}_{Y}^{d}$ and the projection 
$\mathbf{A}_{Y}^{d}\rightarrow Y$.
Hence we may assume that $f$ is a closed immersion or a smooth morphism of relative dimension $d$.
Suppose that $f$ is a closed immersion and 
let $j\colon Y\setminus X\rightarrow Y$ denote the open immersion.
Since we have $f^{!}=f^{*}-f^{*}\circ j_{*} \circ j^{*}$, the assertion follows by (i) and (ii).
Suppose that $f$ is a smooth morphism of relative dimension $d$.
Since we have an isomorphism $t_{f}\colon f^{*}(d)[2d]\rightarrow Rf^{!}$ (\cite[3.2]{sga4}), the assertion follows by (i).

We prove (v).
Let $h\colon X\rightarrow S$ be the structure morphism.
Let $(a,b)$ be an element of $\Delta_{c}(X,\Lambda,\Lambda')$.
By devissage, we may assume that $(a,b)\in \Delta_{coh}(X,\Lambda,\Lambda')$
and that $(h^{!}\Lambda,h^{!}\Lambda')\in K_{coh}(S,\Lambda)\times K_{coh}(S,\Lambda')$.
Since $(\Lambda,\Lambda')\in \Delta_{coh}(X,\Lambda,\Lambda')$, 
we have $(h^{!}\Lambda,h^{!}\Lambda')\in \Delta_{coh}(X,\Lambda,\Lambda')$ by (iv).
Since $\btr_{D_{X}(a)}(g)$ for $g\in E_{X/S}$ is the product of $\btr_{h^{!}\Lambda}(g)$
and the conjugate of $\btr_{a}(g)$ and similarly for $\btr_{D_{X}(b)}(g)$ for $g\in E_{X/S}$,
it is sufficient to prove that $\btr_{h^{!}\Lambda}(g)$ and $\btr_{h^{!}\Lambda'}(g)$ are integers for
every $g\in E_{X/S}$.

As in the proof of (iv), we may assume that $h$ is a closed immersion or a smooth morphism of
relative dimension $d$ for some $d\in \mathbf{Z}_{\ge 0}$.
Suppose that $h$ is a closed immersion.
Let $j\colon S\setminus X\rightarrow S$ be the open immersion.
Since $h^{!}=h^{*}-h^{*}\circ j_{*}\circ j^{*}$, the assertion follows by
applying \cite[Exercise 18.2]{se} to $j_{*}\Lambda$ and $j_{*}\Lambda'$.
Suppose that $h$ is a smooth morphism of relative dimension $d$.
Since we have the isomorphism $t_{h}\colon h^{*}(d)[2d]\rightarrow Rh^{!}$,
the assertion follows.
\end{proof}

\begin{cor}[cf.\ {\cite[Corollaire 0.2]{vi2}}]
\label{corkzt}
Let $f\colon X\rightarrow Y$ be an $S$-morphism 
of separated schemes over $S$ of finite type.
\begin{enumerate}
\item $f^{*}\colon K_{c}(Y,\Lambda)\rightarrow K_{c}(X,\Lambda)$
induces $f^{*}\colon K_{c}(Y,\Lambda)_{0}\rightarrow K_{c}(X,\Lambda)_{0}$.
\item $f_{*}\colon K_{c}(X,\Lambda)\rightarrow K_{c}(Y,\Lambda)$
induces $f_{*}\colon K_{c}(X,\Lambda)_{0}\rightarrow K_{c}(Y,\Lambda)_{0}$.
\item $f_{!}\colon K_{c}(X,\Lambda)\rightarrow K_{c}(Y,\Lambda)$
induces $f_{!}\colon K_{c}(X,\Lambda)_{0}\rightarrow K_{c}(Y,\Lambda)_{0}$.
\item $f^{!}\colon K_{c}(Y,\Lambda)\rightarrow K_{c}(X,\Lambda)$
induces $f^{!}\colon K_{c}(Y,\Lambda)_{0}\rightarrow K_{c}(X,\Lambda)_{0}$.
\item $D_{X}\colon K_{c}(X,\Lambda)\rightarrow K_{c}(X,\Lambda)$
induces $D_{X}\colon K_{c}(X,\Lambda)_{0}\rightarrow K_{c}(X,\Lambda)_{0}$.
\end{enumerate}
\end{cor}

\begin{proof}
Since $\Delta_{c}(X,\Lambda,\Lambda)\cap (K_{c}(X,\Lambda)\times \{0\})$ and
$\Delta_{c}(Y,\Lambda,\Lambda)\cap (K_{c}(Y,\Lambda)\times \{0\})$ are isomorphic to
$K_{c}(X,\Lambda)_{0}$ and $K_{c}(Y,\Lambda)_{0}$ respectively by the first projections,
the assertions (i)--(v) follows from Corollary \ref{corswr} (i)--(v) respectively.
\end{proof}

Vidal's subgroup $K_{c}(X,\Lambda)^{0}_{t}$ of $K_{c}(X,\Lambda)$ consists 
of the elements $a\in K_{c}(X,\Lambda)$
such that there exists a finite decomposition $X=\coprod_{i}X_{i}$ of $X$
into normal connected locally closed subschemes $\{X_{i}\}_{i}$ of $X$
such that $a|_{X_{i}}\in K_{coh}(X_{i},\Lambda)$ for every $i$ and that
$\btr_{a|_{X_{i}}}(g)=0$ for every $i$ and $g\in E_{X_{i}/S}$ 
(\cite[D\'{e}finition 2.3.1]{vi1}).
For Vidal's subgroup $K_{c}(X,\Lambda)^{0}_{t}$, the same assertions in Corollary \ref{corkzt}
hold and further
the compatibility with the derived tensor product and $R\mhom$ hold
by \cite[Corollaire 0.2]{vi2}.
Namely  
$K_{c}(X,\Lambda)^{0}_{t}$ is an ideal of $K_{c}(X,\Lambda)$ with respect to the multiplication induced by the derived tensor product and 
$R\mathcal{H}om\colon K_{c}(X,\Lambda)\times K_{c}(X,\Lambda)\rightarrow 
K_{c}(X,\Lambda)$ induces $R\mathcal{H}om\colon K_{c}(X,\Lambda)^{0}_{t}\times 
K_{c}(X,\Lambda)\rightarrow K_{c}(X,\Lambda)^{0}_{t}$ and 
$R\mathcal{H}om\colon K_{c}(X,\Lambda)\times K_{c}(X,\Lambda)^{0}_{t}\rightarrow 
K_{c}(X,\Lambda)^{0}_{t}$.
However, the compatibility with the derived tensor product or $R\mhom$
does not hold for $K_{c}(X,\Lambda)_{0}$ in general.

\begin{exa}
Let $G$ be a finite group $\mathbf{F}_{p}$.
Let $M$ and $N$ be $1$-dimensional representations of $G$ over $\Lambda$.
Let $m$ and $n$ be bases of $M$ and $N$ respectively.
Assume that $\Lambda$ has a $p$-th root of unity $\zeta_{p}$ not equal to $1$
and that $\zeta_{p}^{-1}\neq \zeta_{p}$.
Further assume that the action of $1\in \mathbf{F}_{p}$ on $M$ is given by $1\cdot m=\zeta_{p}
\cdot m$ and the action of $1\in \mathbf{F}_{p}$ on $N$ is given by $1\cdot n=\zeta^{-1}_{p}\cdot n$.
Then we have $\dim M^{g}=\dim N^{g}$ for every $g\in G$.
However, we have $\dim (M\otimes_{\Lambda}M)^{1}=0$
and $\dim (N\otimes_{\Lambda}M)^{1}=1$.
\end{exa}


\end{document}